\documentclass[11pt]{article}
\def\thetitle{Discrepancies of subtrees}

\usepackage{graphicx}

\usepackage{amsmath,amssymb}

\usepackage[usenames,dvipsnames,svgnames,table]{xcolor}
\definecolor{CombinatoricaAqua}{HTML}{00698C}
\definecolor{CombinatoricaBlue}{HTML}{3A3293}
\definecolor{CombinatoricaBrown}{HTML}{66220C}
\definecolor{CombinatoricaRed}{HTML}{DF2A27}
\definecolor{HarvardCrimson}{rgb}{0.6471, 0.1098, 0.1882}

\makeatletter
\let\reftagform@=\tagform@
\def\tagform@#1{\maketag@@@
	{(\ignorespaces\textcolor{CombinatoricaBrown}{#1}\unskip\@@italiccorr)}}
\renewcommand{\eqref}[1]{\textup{\reftagform@{\ref{#1}}}}
\makeatother

\usepackage[backref=page]{hyperref}
\hypersetup{%
	unicode,
	pdfencoding=auto,
	pdfauthor={Peleg Michaeli},
	pdftitle={\thetitle},
	pdfsubject={},
	pdfkeywords={},
	colorlinks=true,
	citecolor=CombinatoricaBlue,
	linkcolor=CombinatoricaAqua,
	anchorcolor=CombinatoricaBrown,
	urlcolor=HarvardCrimson}

\usepackage{amsthm}
\usepackage{thmtools}
\usepackage{bbm}
\usepackage{enumitem}
\usepackage[capitalize]{cleveref}
\Crefname{fact}{Fact}{Facts}
\Crefname{claim}{Claim}{Claims}


\declaretheoremstyle[
spaceabove=\topsep, spacebelow=\topsep,
headfont=\color{CombinatoricaBrown}\normalfont\bfseries,
bodyfont=\itshape,
]{thm}
\declaretheoremstyle[
spaceabove=\topsep, spacebelow=\topsep,
headfont=\color{CombinatoricaBrown}\normalfont\bfseries,
bodyfont=\normalfont,
]{dfn}
\declaretheoremstyle[
spaceabove=0.5\topsep, spacebelow=0.5\topsep,
headfont=\color{CombinatoricaBrown}\normalfont\bfseries,
bodyfont=\normalfont,
]{rmk}

\declaretheorem[style=thm]{theorem}
\declaretheorem[style=thm,sibling=theorem]{lemma}
\declaretheorem[style=thm,sibling=theorem]{corollary}
\declaretheorem[style=thm,sibling=theorem]{claim}
\declaretheorem[style=thm,sibling=theorem]{proposition}

\declaretheorem[style=thm,sibling=theorem]{conjecture}

\usepackage[nobysame,msc-links]{amsrefs}

\renewcommand{\eprint}[1]{\href{https://arxiv.org/abs/#1}{arXiv:#1}}

\BibSpecAlias{misc}{webpage}
\BibSpec{book}{%
	+{}  {\PrintPrimary}                {transition}
	+{,} { \textbf}                     {title} 
	+{.} { }                            {part}
	+{:} { \textit}                     {subtitle}
	+{,} { \PrintEdition}               {edition}
	+{}  { \PrintEditorsB}              {editor}
	+{,} { \PrintTranslatorsC}          {translator}
	+{,} { \PrintContributions}         {contribution}
	+{,} { }                            {series}
	+{,} { \voltext}                    {volume}
	+{,} { }                            {publisher}
	+{,} { }                            {organization}
	+{,} { }                            {address}
	+{,} { \PrintDateB}                 {date}
	+{,} { }                            {status}
	+{}  { \parenthesize}               {language}
	+{}  { \PrintTranslation}           {translation}
	+{;} { \PrintReprint}               {reprint}
	+{.} { }                            {note}
	+{.} {}                             {transition}
	+{}  {\SentenceSpace \PrintReviews} {review}
}
\BibSpec{incollection}{%
  +{}  {\PrintAuthors}                {author}
  +{,} { \textit}                     {title}
  +{.} { }                            {part}
  +{:} { \textit}                     {subtitle}
  +{,} { \PrintContributions}         {contribution}
  +{,} { \PrintConference}            {conference}
  +{}  {\PrintBook}                   {book}
  +{,} { }                            {booktitle}
	+{}  { \PrintEditorsB}              {editor}
	+{,} { }                            {publisher}
  +{,} { \PrintDateB}                 {date}
  +{,} { pp.~}                        {pages}
  +{,} { }                            {status}
  +{,} { \PrintDOI}                   {doi}
  +{,} { available at \eprint}        {eprint}
  +{}  { \parenthesize}               {language}
  +{}  { \PrintTranslation}           {translation}
  +{;} { \PrintReprint}               {reprint}
  +{.} { }                            {note}
  +{.} {}                             {transition}
  +{}  {\SentenceSpace \PrintReviews} {review}
}

\makeatletter
\renewcommand{\PrintNames@a}[4]{%
	\PrintSeries{\name}
	{#1}
	{}{ and \set@othername}
	{,}{ \set@othername}
	{}{ and \set@othername}
	{#2}{#4}{#3}%
}
\makeatother

\makeatletter
\def\mathcolor#1#{\@mathcolor{#1}}
\def\@mathcolor#1#2#3{%
	\protect\leavevmode
	\begingroup
	\color#1{#2}#3%
	\endgroup
}
\makeatother
\definecolor{Red}{rgb}{0.618,0,0}
\definecolor{Blue}{rgb}{0,0,1}
\definecolor{Green}{rgb}{0,0.298,0}

\newcommand{\Sp}{\mathrm{Sp}}

\usepackage{sectsty}
\sectionfont{\color{CombinatoricaBrown}}
\subsectionfont{\color{CombinatoricaBrown}}
\subsubsectionfont{\color{CombinatoricaBrown}}

\usepackage{soul}
\soulregister\ref7

\usepackage{pifont}
\usepackage{calc}

\usepackage[a4paper]{geometry}
\title{\thetitle}
\usepackage{authblk}
\author{Tarun Krishna}
\author{Peleg Michaeli\thanks{Research partially supported by NSF grant DMS1952285.}}
\author{Michail Sarantis\thanks{Research partially supported by the Onassis Foundation -- Scholarship F ZP 051-1/2019-2020.}}
\author{Fenglin Wang}
\author{Yiqing Wang}
\affil{
  Department of Mathematical Sciences\\
  Carnegie Mellon University\\
  Pittsburgh PA 15213
}

\makeatletter
\def\namedlabel#1#2{\begingroup
  #2%
  \def\@currentlabel{#2}%
  \phantomsection\label{#1}\endgroup
}
\makeatother

\usepackage{mleftright}
\mleftright

\newcommand{\defn}[1]{{\bfseries #1}}

\renewcommand{\phi}{\varphi}
\newcommand{\CC}{\mathbb{C}}

\newcommand{\RR}{\mathbb{R}}
\renewcommand{\SS}{\mathbb{S}}

\newcommand{\cX}{\mathcal{X}}

\newcommand{\cS}{\mathcal{S}}
\newcommand{\cT}{\mathcal{T}}

\renewcommand{\b}{\vect{b}}
\renewcommand{\d}{\vect{d}}

\newcommand{\m}{\vect{m}}
\newcommand{\n}{\vect{n}}

\newcommand{\ceil}[1]{\left\lceil{#1}\right\rceil}

\newcommand{\mon}{\mathfrak{i}}
\DeclareMathOperator{\Beta}{B}

\newcommand{\sm}{\smallsetminus}

\newcommand{\vect}{\mathbf}

\newcommand{\E}[0]{\mathbb{E}}

\newcommand{\Dist}[1]{\mathsf{#1}}

\newcommand{\Unif}{\Dist{Unif}}

\usepackage{tabto}

\newcommand{\cH}{\mathcal{H}}
\newcommand{\cV}{\mathcal{V}}
\newcommand{\cE}{\mathcal{E}}

\newcommand{\D}{\mathcal{D}}
\newcommand{\lD}{\D^{\circ}}
\newcommand{\sD}{\dot\D}
\newcommand{\oD}{\vec\D}

\usepackage{appendix}

\usepackage{subcaption}
\usepackage{tikz}
\usetikzlibrary{calc,arrows.meta}

\begin{document}
\maketitle

\begin{abstract}
  We study multicolour, oriented and high-dimensional discrepancies of the set of all subtrees of a tree.
  As our main result,
  we show that the $r$-colour discrepancy of the subtrees
  of any tree is a linear function of the number of leaves $\ell$ of that tree.
  More concretely, we show that it is bounded by $\ceil{(r-1)\ell/r}$ from below
  and $\ceil{(r-1)\ell/2}$ from above,
  and that these bounds are asymptotically sharp.
  Motivated by this result,
  we introduce natural notions of oriented and high-dimensional discrepancies
  and prove bounds for the corresponding discrepancies of the set of all subtrees of a given tree
  as functions of its number of leaves.
\end{abstract}

\section{Introduction}
Given a hypergraph $\cH=(\cV,\cE)$,
a (two-)colouring of (the vertices of) $\cH$
is a function $f:\cV\to\{\pm 1\}$.
For a hyperedge $A$ we set $f(A)=\sum_{a\in A} f(a)$,
and $|f(A)|$ is called the \defn{imbalance} of $A$.
The (combinatorial) \defn{discrepancy} of $\cH$ is defined to be
\[
  \D(\cH) = \min_{f:\cV\to\{\pm 1\}} \max_{A\in\cE} |f(A)|.
\]
Namely, the discrepancy of $\cH$ is the maximum imbalance of an edge
under an optimal colouring.
It is often convenient to think about this definition in terms of a game:
an adversary colours $\cV$ using $2$ colours.
He tries to do it as balanced as possible, that is,
so that the distribution of the colours in every member of $\cE$ will be as close as possible to uniform.
Our goal is then to find a member of $\cE$ of maximum imbalance.
Over the last century,
the study of discrepancy-type problems has developed into a field
with extensive range and variety,
demonstrating strong ties to number theory, Ramsey theory, and computational methods.
We refer the reader to the book of Matou\v{s}ek~\cite{Mat}
for a comprehensive overview of the topic. 

There are several natural ways to generalise the above definition of ($2$-colour) discrepancy
to an arbitrary number of colours.
One such generalisation was introduced by Doerr and Srivastav~\cite{DS03},
in which the notion of imbalance captures the maximum deviation of the size of a colour class
from the mean size of a colour class
(or, in other words, the (scaled) $\ell^\infty$-distance of the colour distribution from the uniform distribution).
We call it here the \defn{symmetric $r$-colour discrepancy} of $\cH$, and denote\footnote{%
The original definition of Doerr and Srivastava was a $(1/r)$-scaling of the above definition;
we scaled it for convenience to allow $\lD_2=\D$, and to ensure it is an integer.}
\[
  \lD_r(\cH) = \min_{f:\cV\to[r]} \max_{A\in\cE} \max_{j\in[r]}
  \left|r\left|f^{-1}(j)\cap A\right|-|A|\right|.
\]
Recently, mostly in the context of graphs,
a slightly different notion of multicolour discrepancy was studied,
in which the notion of imbalance captures the deviation of the size of the {\em largest} colour class
from the mean size.
We call it here the (upper) \defn{$r$-colour discrepancy} of $\cH$, and denote
\[
  \D_r(\cH) = \min_{f:\cV\to[r]} \max_{A\in\cE} \max_{j\in[r]}
  \left(r\left|f^{-1}(j)\cap A\right|-|A|\right).
\]
It is not hard to see
that these definitions are both generalisations of the classical notion of discrepancy,
and differ from each other by a constant factor.
Concretely,
$\D_2(\cH) = \lD_2(\cH) = \D(\cH)$ and
\begin{equation}\label{eq:mclr}
  \D_r(\cH) \le \lD_r(\cH) \le (r-1)\D_r(\cH)
\end{equation}
for every hypergraph $\cH$ and $r\ge 2$.

The ``upper'' variation is more natural in the context of edge-colourings in (hyper)graphs,
due to its direct relation to Ramsey-type questions:
given an edge-colouring of a graph,
instead of looking for a monochromatic copy of a target subgraph,
one looks for a copy of that subgraph in which one of the colours appears (significantly) more than the average.
In that sense, discrepancy-type problems may be considered as a relaxation --- or rather a quantification ---
of Ramsey-type problems.

Let us elaborate on combinatorial discrepancies in the context of graphs.
Here,
given a base graph $G$ and a family of graphs $\cX$,
we construct a hypergraph whose vertices are the edges of $G$
and whose hyperedges are edge sets that from a member of $\cX$.
The \defn{discrepancy of $\cX$ in $G$}, denoted $\D(G,\cX)$,
is the discrepancy of that hypergraph.
Analogously, we define the
symmetric $r$-colour discrepancy of $\cX$ in $G$ ($\lD_r(G,\cX)$)
and
the $r$-colour discrepancy of $\cX$ in $G$ ($\D_r(G,\cX)$).
It is helpful to keep in mind that $\D(\cH)$ is monotone in $\cE$,
hence $\D(G,\cX)$ is monotone both in $G$ and in $\cX$.
It is therefore natural (and often nontrivial) to study $\D(K_n,\cX)$.

The study of combinatorial discrepancy in graphs
was initiated by Erd\H{o}s, F\"uredi, Loebl and S\'os~\cite{EFLS95},
who analysed the $2$-colour discrepancy of a fixed spanning tree with a given maximum degree in the complete graph.
However, several earlier results can be stated using this terminology.
As an example we mention the result of
Erd\H{o}s and Spencer~\cite{ES72},
that can be interpreted as showing that the ($2$-colour) discrepancy of cliques in the complete graph on $n$ vertices
(or, more generally, of hypercliques in the complete $k$-uniform hypergraph) is of order $n^{3/2}$
(or, more generally, $n^{(k+1)/2}$).
Recently, Balogh, Csaba, Jing and Pluh\'ar~\cite{BCJP20} initiated the study of discrepancies in general graphs.
In particular, they obtained a Dirac-type bound for positive discrepancy of Hamilton cycles
(in $2$ colours; this was generalised to $r$ colours in~\cite{FHLT21} and independently in~\cite{GKM22b}),
and estimated the discrepancy of the set of all spanning trees in random regular graphs
and $2$-dimensional grids (in $2$ colours).
The last result was greatly generalised to $r$ colours and to almost every base graph in~\cite{GKM22b},
where the authors establish a non-trivial connection between the spanning-tree discrepancy
(essentially an extremal quantity)
and a purely geometric property of the graph.
In $2$-dimensional grids,
Balogh et al.\ also showed that the discrepancy of paths (hence also of trees) is linear in the number of vertices.

Other recent works include
an estimate of multicolour discrepancy in random graphs and in the complete graph~\cite{GKM22a};
a Dirac-type bound for positive $2$-colour discrepancy of $k$-factors~\cite{BCPT21};
and a Dirac-type bound for positive $2$-colour discrepancy of powers of Hamilton cycles~\cite{Bra22}.
Finally, Gishboliner, Krivelevich and the second author have introduced a notion of {\em oriented} discrepancy,
and studied the oriented discrepancy of Hamilton cycles in dense and in random graphs~\cite{GKM22+}.
We will elaborate on this matter further.

The present work continues this line of research.
Our main result shows that the $r$-colour discrepancy of the set of all trees in a given tree
is linear in the number of leaves of that tree.
Let us denote the set of all trees by $\cT$.
Thus, for a graph $G$, $\D_r(G, \cT)$ denotes the $r$-colour discrepancy of trees in $G$.
For a tree $T$, denote by $\ell(T)$ the number of leaves in $T$.
As a warm-up example, consider the following two simple cases.
Let $S_\ell$ denote the star with $\ell$ leaves.
It is evident that an optimal colouring is an equipartition of the leaves into the $r$ colour classes,
and the most unbalanced tree in this case will be a monochromatic substar.
Hence,
\begin{equation}\label{eq:disc:stars}
  \D_r(S_\ell,\cT)
  = r\cdot\ceil{\frac{\ell}{r}} - \ceil{\frac{\ell}{r}}
  =(r-1)\ceil{\frac{\ell}{r}}.
\end{equation}
Similarly, considering the path $P_n$ on $n$ vertices
(so $\ell(P_n)=2$),
an optimal colouring can easily be seen to be any periodic colouring,
in which the most unbalanced tree will be a single edge.
Hence,
\begin{equation}\label{eq:disc:paths}
  \D_r(P_n,\cT)
  = r\cdot 1 - 1
  =(r-1)\ceil{\frac{2}{r}}.
\end{equation}

Given \cref{eq:disc:stars,eq:disc:paths}, a natural guess would be that any tree $T$ satisfies
$\D_r(T,\cT)=(r-1)\ceil{\ell(T)/r}$.
It turns out that the above holds for $r=2$ (see below).
For $r\ge 3$, however, this is only (at least asymptotically) a lower bound,
and the star demonstrates that it is sharp.
Our first and main result gives bounds on $\D_r(T,\cT)$ in terms of $\ell(T)$.
\begin{theorem}[Multicolour discrepancy]\label{thm:mclr}
  For every $r\ge 2$ and every tree $T$ with $\ell$ leaves,
  \[
    \ceil{(r-1)\cdot\frac{\ell}{r}} \le \D_r(T,\cT) \le \ceil{(r-1)\cdot\frac{\ell}{2}}.
  \]
  In particular, for $r=2$ we have
\[
  \D_2(T,\cT) = \ceil{\frac{\ell}{2}}.
\]
\end{theorem}
We explained earlier why the lower bound in \cref{thm:mclr} is sharp
(asymptotically and for infinitely many values of $\ell$).
In \cref{sec:mclr}, where we prove the theorem,
we also prove that the upper bound is sharp
(exactly and for every $\ell$; see \cref{prop:mclr:ub:sharp}).

Using a classical result of Kleitman and West~\cite{KW91}
about the maximum number of leaves in a spanning tree of a graph
(sometimes called the \defn{maximum leaf number}),
we obtain the following
improvement\footnote{%
Their result, for $r=2$ only, is an immediate corollary of a stronger result they prove
on the discrepancy of {\em paths} in the grid.
On the other hand,
while their proof is a clever ad-hoc and suited for grids,
our proof is more general.%
} and extension (to any number of colours)
of \cite{BCJP20}*{Corollary~7}.
For a proof, see \cref{sec:grid}.
\begin{corollary}\label{cor:mclr:grid}
  Let $m,n\ge 2$ be integers
  and let $G$ be the $m\times n$ grid.
  Then
  \[\D_r(G,\cT)\ge \frac{r-1}{4r}\cdot mn+1-2r.\]
\end{corollary}

Our next result is in the context of signed/oriented discrepancy.
Let us lay a formal ground to state our results.
The notion of a \defn{signed hypergraph}, introduced by Shi~\cite{Shi92},
is an extension of the conventional notion of a hypergraph
that allows ``negative'' vertex-edge incidences.
Formally, a signed hypergraph $\cH$ is a triple $(\cV,\cE,\psi)$
where $\cV,\cE$ are disjoint sets (``vertices'' and ``hyperedges'')
and $\psi:\cV\times\cE\to\{-1,0,1\}$ is an \defn{incidence function}.
For a hyperedge $A$ we set $f(A)=\sum_{a\in\cV}f(a)\cdot\psi(a,A)$,
and $|f(A)|$ is called the \defn{imbalance} of $A$.
We define the \defn{signed discrepancy} of $\cH$ to be
\[
  \sD(\cH) = \min_{f:\cV\to\{\pm 1\}} \max_{A\in\cE} \left|\sum_{a\in\cV} f(a)\cdot \psi(a,A)\right|.
\]
With a slight abuse, we may ignore the formal definition that contains the incidence function,
and instead think of sets in a more general way:
for each set and each element, the set can contain the element,
not contain the element, or ``negatively'' contain that element.
This notion turns out to be useful in many cases, as we will see below.
Note that $\sD(\cH)=\D_2(\cH)$ if $\psi$ is nonnegative;
in that sense, the signed discrepancy is a direct generalisation of $2$-colour discrepancy.

Analogously to how we defined multicolour discrepancies in graphs,
we define oriented discrepancy in graphs.
In this setting, given an oriented\footnote{That base orientation will not matter and can be arbitrary.}
base graph $G$
and a family of oriented graphs $\cX$,
we construct a signed hypergraph whose vertices are the edges of $G$
and whose hyperedges are edge sets that from a member of $\cX$,
where an edge is positively contained in a hyperedge if its orientation in $G$ agrees with its orientation in $\cX$,
and negatively contained otherwise.
The \defn{oriented discrepancy of $\cX$ in $G$},
denoted $\oD(G,\cX)$,
is the signed discrepancy of that signed hypergraph\footnote{%
A potential term would have been \defn{signed discrepancy};
however, when the vertices of the hypergraph represent edges of a graph,
the notion of orientation is more natural.}.
Again, it is convenient to think about this definition in terms of a game:
an adversary orients the edges of $G$ (ignoring the ``original'' orientation it had).
He tries to do it as balanced as possible, that is,
so that in any member of $\cX$,
the number of edges in which the orientation in $\cX$ agrees with his orientation of $G$
is as close to 50\% as possible.
Our goal is then to find a member of $\cX$ of maximum imbalance,
namely, that contains many more agreements than disagreements, or the other way around.

Let $\mathcal{DHAM}$ be the set of all {\em directed} Hamilton cycles.
The result of \cite{GKM22+} on the oriented discrepancy of Hamilton cycles in Dirac graphs
can be restated as follows:
if $G$ is an $n$-vertex graph with $\delta(G)\ge n/2+8$
then $\oD(G,\mathcal{DHAM})=\Omega(2\delta(G)-n))$.
The authors of \cite{GKM22+} conjectured that if $\delta(G)\ge n/2$
then $\oD(G,\mathcal{DHAM})\ge 2\delta(G)-n$,
and that if true, it would be best possible.
The conjecture --- a strong generalisation of Dirac's theorem
--- was fully resolved by Freschi and Lo~\cite{FL23+}:
\begin{theorem}[\cite{FL23+}*{Theorem 1.5}]
  Let $G$ be an $n$-vertex graph with $n\ge 3$ $\delta(G)\ge n/2$.
  Then, $\oD(G,\mathcal{DHAM})\ge 2\delta(G)-n$.
\end{theorem}

Here, we obtain a new result in the setting of oriented discrepancy in graphs.
Let $\mathcal{DT}$ denote the set of all directed rooted trees;
namely, trees that have a distinguished vertex called the \defn{root}
and that are oriented {\em away} from that root%
\footnote{This is an arbitrary choice of one of two natural orientations of a rooted tree,
and has no implications on the results.}.
Our next theorem gives bounds on $\oD(T,\mathcal{DT})$ in terms of $\ell(T)$.
\begin{theorem}[Oriented discrepancy]\label{thm:orient}
  For every tree $T$ on at least $3$ vertices and with $\ell$ leaves,
  \[
    \ceil{\frac{\ell}{2}}+1 \le \oD(T,\mathcal{DT}) \le \ell.
  \]
\end{theorem}
The lower bound is sharp
(exactly and for every $\ell$),
since a star with $\ell$ leaves
that is oriented as evenly as possible has oriented imbalance $\ceil{\ell/2}+1$
(see \cref{prop:odisc:star}).
We conjecture that one can obtain a better upper bound
that matches the lower bound asymptotically.

\begin{conjecture}\label{conj:orient}
  For every tree $T$ with $\ell$ leaves,
  $\oD(T,\mathcal{DT})=(1+o(1))\ell/2$.
\end{conjecture}

The multicolour discrepancy $\D_r$ and the signed discrepancy $\sD$
are two natural generalisations of the classical notion of discrepancy $\D$,
both of combinatorial nature.
In some sense, however, they lack the geometric aspect of discrepancy.
In particular, $\D_r$ is not even generally monotone in $r$.
Following Tao~\cite{Tao16},
we may generalise the definition of discrepancy geometrically,
by allowing vector-valued colouring functions.
Here, we restrict our attention to the (already challenging) case of the vector space $\RR^d$.
For $d\ge 0$, let $\SS^d$ denote the $d$-dimensional unit hypersphere in $\RR^{d+1}$.
A $d$-dimensional colouring of $\cH$
is a function $f:\cV\to\SS^d$.
For a hyperedge $A$ we set $f(A)=\sum_{a\in A} f(a)$,
and $|f(A)|$ is called the \defn{imbalance} of $A$.
We define the \defn{$d$-dimensional discrepancy}
of $\cH$ to be
\[
  \D^d(\cH) = \min_{f:\cV\to\SS^d} \max_{A\in\cE} |f(A)|.
\]
We observe that $\D^0=\D_2$,
and that $\D^{d'}\le\D^d$ whenever $d'\ge d$.
Understanding $\D^0$ quite well, we move on to study $\D^d$ for $d\ge 1$.

While the $r$-colour discrepancy is more combinatorial in nature
and the $d$-dimensional discrepancy is more geometric,
the following proposition relates the two notions.
\begin{proposition}\label{prop:complex}
  For every $r\ge 2$, $d\ge 1$ and hypergraph $\cH$,
  $\D^d(\cH)\le \D^1(\cH) \le \lD_r(\cH)$.
\end{proposition}
We prove this proposition in~\cref{sec:highdim}.
It follows from \cref{eq:mclr,thm:mclr,prop:complex}
that the high-dimensional tree-discrepancy of a tree with $\ell$ leaves
is at most $\ceil{\ell/2}$ (for every dimension $d\ge 1$).
In the next theorem we prove a lower bound
that we believe that under some assumptions matches the upper bound (see \cref{conj:complex:lb}).
Let $\Beta(z_1,z_2)$ be the \defn{beta function},
and recall that $\Beta(z_1,z_2)=\Gamma(z_1)\Gamma(z_2)/\Gamma(z_1+z_2)$,
where $\Gamma$ is the \defn{gamma function}.
\begin{theorem}[high-dimensional discrepancy]\label{thm:highdim:lb}
  For every tree $T$ with $\ell$ leaves,
  \[
    \D^d(T,\cT) \ge \frac{\ell}{d\cdot \Beta\left(\frac{d}{2},\frac{1}{2}\right)}.
  \]
  In particular, $\D^1(T,\cT)\ge \ell/\pi$ and
  $\D^d(T,\cT)\ge (1-o_d(1))\ell/\sqrt{2\pi d}$.
\end{theorem}

When $d=1$, it is convenient to identify $\RR^2$ with $\CC$
and think of the colours as complex numbers on the unit circle.
We therefore refer to the $1$-dimensional discrepancy as \defn{complex discrepancy}.
We conclude with a conjecture about complex discrepancy.
\begin{conjecture}\label{conj:complex:lb}
  For every tree $T$ with $\ell$ leaves,
  \[
    \D^1(T,\cT)
    \ge \frac{1}{2\sin\left(\frac{\pi}{2\ell}\right)}
    = (1+o(1))\frac{\ell}{\pi}.
  \]
  If, in addition, $\Delta(T)=\omega(1)$,
  then
  \[
    \D^1(T,\cT) = (1+o(1))\frac{\ell}{\pi}.
  \]
  
\end{conjecture}

\section{Multicolour discrepancy}\label{sec:mclr}
In this section we prove \cref{thm:mclr} and its sharpness.
Given a tree $T$, a colouring $f:E(T)\to[r]$ and a subtree $S$,
write $e_j(S)=|f^{-1}(S)|$ and $w_j(S)=re_j(S)-|E(S)|$.

\begin{proof}[Proof of the lower bound in \cref{thm:mclr}]
  Let $T$ be a tree with $\ell$ leaves,
  and let $f:E(T)\to[r]$ be an $r$-colouring of its edges.
  Denote $m_j=|f^{-1}(j)|$ for $j\in[r]$ and $m=\sum_{j\in[r]}m_j=|E(T)|$.
  We obtain a subtree $T'$ of $T$ be deleting all leaves of $T$,
  and denote $m'_j=|f^{-1}(j)\cap E(T')|$
  and $\ell_j=|f^{-1}(j)\sm E(T')|$
  for $j\in[r]$.
  Write $m'=\sum_{j\in[r]}m'_j$
  and note that $\sum_{j\in[r]}\ell'_j=m-m'=\ell$.
  Finally, we obtain $T_j$ from $T'$ by adding back the edges $f^{-1}(j)\sm E(T')$.
  Observe that
  $w_j(T_j) = rm_j-m'-\ell_j$,
  hence
  \[
    \sum_{j\in[r]} w_j(T_j) = rm-rm'-\ell = (r-1)\ell.
  \]
  Thus, by the pigeonhole principle, there exists $j\in[r]$ for which $w_j\ge\ceil{(r-1)\ell/r}$,
  hence $\D_r(T,\cT)\ge\ceil{(r-1)\ell/r}$.
\end{proof}

We move on to prove the upper bound in \cref{thm:mclr}.
  Consider the pointwise partial order relation on $\RR^r$
  defined as follows:
  for $\m=(m_1,\ldots,m_r)$ and $\n=(n_1,\ldots,n_r)$,
  $\m\le\n$ if and only if $m_j\le n_j$ for every $j\in[r]$.
  For a permutation $\tau$ of $[r]$
  we write $\tau(\m)=(m_{\tau(1)},\ldots,m_{\tau(r)})$.
  We say that $\m$ is \defn{dominated} by $\n$
  and denote it $\m\preceq\n$
  if there exists a permutation $\tau$ of $[r]$ such that
  $\tau(\m)\le\n$.
  We further write $\m\lor\n=(m_1\lor n_1,\ldots,m_r\lor n_r)$,
  where for real numbers $x,y$, $x\lor y=\max\{x,y\}$.
  Call $\m$ is \defn{increasing} if it is (weakly) monotone increasing as a sequence.
  Denote by $\sigma_\m$ the first permutation of $[r]$ (according to some arbitrary fixed ordering)
  for which $\sigma_\m(\m)$ is increasing.
  Write $\mon(\m)=\sigma_\m(\m)$ for the ``monotone version'' of $\m$.
  Let $\min\m$ and $\max\m$
  denote the minimal and maximal coordinate in $\m$, respectively,
  and note that if $\m\preceq\n$ then $\max\m\le\max\n$.
  Say that a vector $\m$ is \defn{$1$-Lipschitz}
  if for every $1\le j<r$, $|\m_{j+1}-\m_j|\le 1$.

  For a vector $\m$,
  let $\alpha_{\m}$ be the vector $(a_1,\ldots,a_r)$
  where $a_j=m_j+r-\sigma_{\m}^{-1}(j)$.
  It is useful to observe that if $\m\preceq\n$
  then $\alpha_{\m}\preceq\alpha_{\n}$.
  It is also useful to observe that if $\m$ is increasing $1$-Lipschitz
  then $\alpha_\m$ is decreasing $1$-Lipschitz.
  For every $j\in[r]$ denote
  $d_j=r-\ceil{(r+1-j)/2}$,
  and let $\d_2=(d_1,\ldots,d_r)$.
  For $\ell\ge 2$ define $\d_{\ell+1}=\mon(\alpha_{\d_\ell})$,
  and note that for $\ell\ge 3$, $\min\d_\ell=\max\d_{\ell-1}$.
  Note futher that since $\d_2$ is increasing $1$-Lipschitz
  then by the discussion above, $\d_\ell$ is increasing $1$-Lipschitz for every $\ell\ge 2$.
  Thus, for every $\ell\ge 3$,
  $(\d_\ell)_j=(\d_{\ell-1})_{r+1-j}+j-1$.
  In particular,
  $\max\d_3 = d_1+r-1=2r-1-\ceil{r/2}=\ceil{3(r-1)/2}$,
  and, for $\ell\ge 4$,
  $\max\d_\ell=\min\d_{\ell-1}+r-1=\max\d_{\ell-2}+r-1$.
  By induction, $\max\d_\ell = \ceil{\ell(r-1)/2}$.
  The following claim will be useful for us.

  \begin{claim}\label{cl:dom}
    For every $\ell\ge 3$,
    if $\m\preceq\d_{\ell-1}$
    and $\n\preceq\d_\ell$
    then $\m\lor\n\preceq\d_\ell$.
  \end{claim}

  \begin{proof}
    We may assume that $\mon(\n)=\d_\ell$.
    Thus, $\min\n=\min\d_\ell=\max\d_{\ell-1}$,
    hence $\m\lor\n=\n$, and the claim follows.
  \end{proof}

\begin{proof}[Proof of the lower bound in \cref{thm:mclr}]
  For vertices $u,v\in V(T)$,
  let $\cS_v(T)$ denote the set of subtrees $S$ of $T$ that contain the vertex $v$,
  let $\cS_{u,v}(T)$ the set of subtrees $S$ that contain $u,v$,
  and let $\cS_{u,\neg v}(T)$ be the set of subtrees $S$ that contain $u$ but not $v$.
  For $j\in[r]$,
  define $M_j(v;T)=\max_{S\in\cS_v(T)} w_j(S)$
  and analogously $M_j(u,v;T)$ and $M_j(u,\neg v;T)$.
  The \defn{colour profile} of $v$ in $T$
  (with respect to a colouring $f$)
  is the vector $\chi(v;T)=(M_1(v;T),\ldots,M_r(v;T))$.
  Define analogously $\chi(u,v;T)$ and $\chi(u,\neg v;T)$.
  We prove by induction the following statement:
  for every $\ell\ge 2$
  and every tree $T$ with $\ell$ leaves,
  there exists a colouring $f$ of $E(T)$
  for which for every vertex $v\in V(T)$,
  $\chi(v;T)\preceq\d_\ell$.
  This would imply, in particular,
  that for every subtree $S$ of $T$
  and every colour $j\in[r]$,
  $w_j(S)\le\max\chi(v;T)$ for some vertex $v\in V(S)$;
  but for every $v\in V(T)$, $\max\chi(v;T)\le\max\d_\ell\le \ceil{(r-1)\ell/2}$,
  implying the statement of the theorem.
  Our inductive argument yields a concrete explicit colouring of $E(T)$;
  see \cref{sec:algo} for an (implied) efficient algorithmic version.

  The base case is when $\ell=2$.
  Here, $T$ is a path; suppose the edges of the path are $(e_1,\ldots,e_k)$ in this order.
  We colour the path periodically;
  namely, we let $f(e_i)=j$ if and only if $i\equiv j\pmod{r}$.
  Let $v\in V(T)$ and let $S\in\cS_v$ be a subpath of $T$ containing $v$.
  Evidently, $w_j(S)\le r-1$ for every $j\in S$.
  Thus, $\chi(v;T)\preceq\d_2$.

  We move on to the induction step.
  Let $T$ be a tree with $\ell=\ell(T)\ge 3$ and suppose the statement holds for $\ell-1$.
  Let $u$ be a leaf in $T$,
  and let $b$ be the \defn{branching vertex} of $u$,
  namely,
  the nearest vertex to $u$ with degree greater than two.
  Let $P_u$ be the path connecting $u$ to $b$,
  and denote by $T'$ the subtree of $T$ obtained by removing all edges of $P_u$
  and all vertices of $P_u$ but $b$.
  Evidently, $\ell(T')=\ell-1$.
  By the induction hypothesis, there exists an $r$-colouring $f'$ of $E(T')$
  that satisfies $\chi(v;T')\preceq\d_{\ell-1}$ for every $v\in V(T')$.
  We extend $f'$ to a colouring $f$ of $E(T)$ as follows.
  Let $\b'=\chi(b;T')$ be the colour profile of $b$ in $T'$.
  Consider the permutation $\sigma_{\b'}$.
  Colour the edges of $P_u$ periodically according to $\sigma_{\b'}$;
  namely, if $P_u=(e_1,\ldots,e_k)$ (where $b\in e_1$ and $u\in e_k$),
  let $f(e_i)=\sigma_{\b'}(j)$ if and only if $i\equiv j\pmod{r}$.
  Note that for any subpath $Q$ of $P_u$ that contains $b$,
  and for any colour $j\in[r]$,
  $w_j(Q)\le r-\sigma_{\b}^{-1}(j)$.
  We now show that $T$ satisfies the hypothesis (with respect to $f$).
  Namely, we show that for every $v\in V(T)$, $\chi(v;T)\preceq\d_\ell$.
  We consider three separate cases.
  \begin{description}
    \item[Case I, $v=b$:]
      We observe that for every $S\in\cS_b$ and every $j\in[r]$,
      letting $S'=S\cap T'$ and $S^-=S\cap P_u$,
      we have $w_j(S)=w_j(S')+w_j(S^-)\le M_j(b;T')+r-\sigma_{\b'}^{-1}(j)$.
      Thus,
      $\chi(b;T)\le\alpha_{\b'}$.
      By the induction hypothesis,
      $\b'\preceq\d_{\ell-1}$,
      hence
      $\chi(b;T)\le\alpha_{\b'}\preceq\mon(\alpha_{\d_{\ell-1}})=\d_\ell$.
    \item[Case II, $v\in V(P_u)\sm\{b\}$:]
      As with the base case of the induction, we have
      $\chi(v,\neg b;T)\le\d_2\le\d_{\ell-1}$.
      On the other hand,
      $\chi(v,b;T)\le\chi(b;T)\preceq\d_\ell$ (by Case I).
      Thus,
      $\chi(v;T)=\chi(v,\neg b;T)\lor\chi(v,b;T)\preceq\d_\ell$
      (by \cref{cl:dom}).
    \item[Case III, $v\in V(T')\sm\{b\}$:]
      By the induction hypothesis
      $\chi(v,\neg b;T)\le\chi(v;T')\preceq\d_{\ell-1}$.
      On the other hand,
      $\chi(v,b;T)\le\chi(b;T)\preceq\d_\ell$.
      Thus,
      $\chi(v;T)=\chi(v,\neg b;T)\lor\chi(v,b;T)\preceq\d_\ell$
      (by \cref{cl:dom}).
  \end{description}
  The proof is now complete.
\end{proof}

\subsection{Algorithmic aspect}\label{sec:algo}
We briefly discuss how the inductive argument presented
in the proof of the upper bound of \cref{thm:mclr}
yields a simple and efficient algorithm for finding a colouring the achieves at least the upper bound.

We begin by describing an efficient algorithm to compute the colour profile of a vertex $b$
in an $r$-coloured tree $T$.
The input is a given tree $T$ with $m$ edges, a vertex $b$, and an $r$-colouring $f$.
For every vertex $v$ of $T$,
let $T_v$ denote the tree rooted at $v$ comprised of $v$ and all its descendants in $T$.
Now observe that, by considering the imbalance of color $j$ at each subtree $T_v$, $v\in N(b)$,
we have
\[
  M_j(b;T) = \sum_{v\in N(b)} \max\{r\cdot \mathbf{1}_{f(\{b,v\})=j}-1+M_j(v;T_v),0\}.
\]
Hence, computing $\chi(b;T)$ requires $O(rm)$ steps.

We proceed by describing the colouring procedure.
We are given a tree $T$ with $m$ edges and $\ell$ leaves,
and a number of colours $r$.
Let $u_1,u_2$ be two distinct leaves of $T$,
and let $P$ be the unique path between them in $T$.
We colour $P$ alternately with a fixed (arbitrary) cyclic order of the colours.
Set $T'=P$.
We then iterate over the remaining $\ell-2$ leaves:
given a leaf $u$ that is not in $T'$,
let $P_u$ be the unique path in $T$ from $u$ to $T'$,
and let $b$ be the last vertex in the path (so $b\in V(T')$).
We can calculate the colour profile $\chi(b;T')$ of $b$ in $T'$ in $O(rm)$ steps.
Given the colour profile,
we colour the path from $b$ to $u$ alternately with a cyclic order of the colours,
from the least popular colour up to the most popular.
That is, the order of colours is $\sigma_{\chi(b;T')}$.
We then add the new coloured path to $T'$ and continue to the next leaf outside $T'$.

This algorithm runs, therefore, in $O(rm\ell)$ steps.
Its correctness was verified recursively in the proof of the upper bound in \cref{thm:mclr}.
See \cref{fig:ub} for a visualisation of the algorithm.

\tikzset{vertex/.style={fill,circle,inner sep=1.5pt},
         fade/.style={black!20!white},
         tvx/.style={vertex,fade},
         ted/.style={very thin,fade},
         edge/.style={very thick},
         er/.style={edge,red},
         eg/.style={edge,ForestGreen,densely dashed},
         eb/.style={edge,blue,densely dotted}
         }
\begin{figure*}[t!]
  \captionsetup{width=0.879\textwidth,font=small}
  \centering
  \begin{subfigure}[t]{0.29\textwidth}
    \centering
    \begin{tikzpicture}[scale=0.5]
      \clip (-0.5,-3.5) rectangle (7.5,2.5);
      \node[vertex] (v1) at (0,0) {};
      \node[vertex] (v2) at (1,0) {};
      \node[vertex] (v3) at (2,0) {};
      \node[vertex] (v4) at (3,0) {};
      \node[vertex] (v5) at (4,0) {};
      \node[vertex] (v6) at (5,0) {};
      \node[vertex] (v7) at (6,0) {};
      \node[tvx] (v8) at (1,-1) {};
      \node[tvx] (v9) at (1,-2) {};
      \node[tvx] (v10) at (3,-1) {};
      \node[tvx] (v11) at (3,-2) {};
      \node[tvx] (v12) at (3,-3) {};
      \node[tvx] (v13) at (5,1) {};
      \node[tvx] (v14) at (5,2) {};
      \node[tvx] (v15) at (6,2) {};
      \node[tvx] (v16) at (4,-2) {};
      \node[tvx] (v17) at (5,-2) {};
      \draw[er] (v1) -- (v2);
      \draw[eg] (v2) -- (v3);
      \draw[eb] (v3) -- (v4);
      \draw[er] (v4) -- (v5);
      \draw[eg] (v5) -- (v6);
      \draw[eb] (v6) -- (v7);
      \draw[ted] (v2) -- (v8);
      \draw[ted] (v8) -- (v9);
      \draw[ted] (v4) -- (v10);
      \draw[ted] (v10) -- (v11);
      \draw[ted] (v11) -- (v12);
      \draw[ted] (v6) -- (v13);
      \draw[ted] (v13) -- (v14);
      \draw[ted] (v14) -- (v15);
      \draw[ted] (v11) -- (v16);
      \draw[ted] (v16) -- (v17);
    \end{tikzpicture}
    \caption{$\ell=2$: a periodic colouring of a path.}
    \label{fig:ub:l2}
  \end{subfigure}%
  ~
  \begin{subfigure}[t]{0.04\textwidth}
    ~
  \end{subfigure}
  \begin{subfigure}[t]{0.29\textwidth}
    \centering
    \begin{tikzpicture}[scale=0.5]
      \clip (-0.5,-3.5) rectangle (7.5,2.5);
      \node[vertex] (v1) at (0,0) {};
      \node[vertex,label={90:$b$}] (v2) at (1,0) {};
      \node[vertex] (v3) at (2,0) {};
      \node[vertex] (v4) at (3,0) {};
      \node[vertex] (v5) at (4,0) {};
      \node[vertex] (v6) at (5,0) {};
      \node[vertex] (v7) at (6,0) {};
      \node[vertex] (v8) at (1,-1) {};
      \node[vertex,label={-90:$u$}] (v9) at (1,-2) {};
      \node[tvx] (v10) at (3,-1) {};
      \node[tvx] (v11) at (3,-2) {};
      \node[tvx] (v12) at (3,-3) {};
      \node[tvx] (v13) at (5,1) {};
      \node[tvx] (v14) at (5,2) {};
      \node[tvx] (v15) at (6,2) {};
      \node[tvx] (v16) at (4,-2) {};
      \node[tvx] (v17) at (5,-2) {};
      \draw[er] (v1) -- (v2);
      \draw[eg] (v2) -- (v3);
      \draw[eb] (v3) -- (v4);
      \draw[er] (v4) -- (v5);
      \draw[eg] (v5) -- (v6);
      \draw[eb] (v6) -- (v7);
      \draw[eb] (v2) -- (v8);
      \draw[er] (v8) -- (v9);
      \draw[ted] (v4) -- (v10);
      \draw[ted] (v10) -- (v11);
      \draw[ted] (v11) -- (v12);
      \draw[ted] (v6) -- (v13);
      \draw[ted] (v13) -- (v14);
      \draw[ted] (v14) -- (v15);
      \draw[ted] (v11) -- (v16);
      \draw[ted] (v16) -- (v17);
    \end{tikzpicture}
    \caption{$\ell=3$:
             $\chi(b;T')=(2,2,1)$.
             We thus colour $P_u$ periodically blue--red--green.}
    \label{fig:ub:l3}
  \end{subfigure}
  \begin{subfigure}[t]{0.04\textwidth}
  ~
  \end{subfigure}
  \begin{subfigure}[t]{0.29\textwidth}
    \centering
    \begin{tikzpicture}[scale=0.5]
      \clip (-0.5,-3.5) rectangle (7.5,2.5);
      \node[vertex] (v1) at (0,0) {};
      \node[vertex] (v2) at (1,0) {};
      \node[vertex] (v3) at (2,0) {};
      \node[vertex,label={90:$b$}] (v4) at (3,0) {};
      \node[vertex] (v5) at (4,0) {};
      \node[vertex] (v6) at (5,0) {};
      \node[vertex] (v7) at (6,0) {};
      \node[vertex] (v8) at (1,-1) {};
      \node[vertex] (v9) at (1,-2) {};
      \node[vertex] (v10) at (3,-1) {};
      \node[vertex] (v11) at (3,-2) {};
      \node[vertex,label={0:$u$}] (v12) at (3,-3) {};
      \node[tvx] (v13) at (5,1) {};
      \node[tvx] (v14) at (5,2) {};
      \node[tvx] (v15) at (6,2) {};
      \node[tvx] (v16) at (4,-2) {};
      \node[tvx] (v17) at (5,-2) {};
      \draw[er] (v1) -- (v2);
      \draw[eg] (v2) -- (v3);
      \draw[eb] (v3) -- (v4);
      \draw[er] (v4) -- (v5);
      \draw[eg] (v5) -- (v6);
      \draw[eb] (v6) -- (v7);
      \draw[eb] (v2) -- (v8);
      \draw[er] (v8) -- (v9);
      \draw[eg] (v4) -- (v10);
      \draw[er] (v10) -- (v11);
      \draw[eb] (v11) -- (v12);
      \draw[ted] (v6) -- (v13);
      \draw[ted] (v13) -- (v14);
      \draw[ted] (v14) -- (v15);
      \draw[ted] (v11) -- (v16);
      \draw[ted] (v16) -- (v17);
    \end{tikzpicture}
    \caption{$\ell=4$:
             $\chi(b;T')=(3,2,3)$.
             We thus colour $P_u$ periodically green--red--blue.}
    \label{fig:ub:l4}
  \end{subfigure}\\\vspace{1em}
  \begin{subfigure}[t]{0.29\textwidth}
    \centering
    \begin{tikzpicture}[scale=0.5]
      \clip (-0.5,-3.5) rectangle (7.5,2.5);
      \node[vertex] (v1) at (0,0) {};
      \node[vertex] (v2) at (1,0) {};
      \node[vertex] (v3) at (2,0) {};
      \node[vertex] (v4) at (3,0) {};
      \node[vertex] (v5) at (4,0) {};
      \node[vertex,label={-90:$b$}] (v6) at (5,0) {};
      \node[vertex] (v7) at (6,0) {};
      \node[vertex] (v8) at (1,-1) {};
      \node[vertex] (v9) at (1,-2) {};
      \node[vertex] (v10) at (3,-1) {};
      \node[vertex] (v11) at (3,-2) {};
      \node[vertex] (v12) at (3,-3) {};
      \node[vertex] (v13) at (5,1) {};
      \node[vertex] (v14) at (5,2) {};
      \node[vertex,label={0:$u$}] (v15) at (6,2) {};
      \node[tvx] (v16) at (4,-2) {};
      \node[tvx] (v17) at (5,-2) {};
      \draw[er] (v1) -- (v2);
      \draw[eg] (v2) -- (v3);
      \draw[eb] (v3) -- (v4);
      \draw[er] (v4) -- (v5);
      \draw[eg] (v5) -- (v6);
      \draw[eb] (v6) -- (v7);
      \draw[eb] (v2) -- (v8);
      \draw[er] (v8) -- (v9);
      \draw[eg] (v4) -- (v10);
      \draw[er] (v10) -- (v11);
      \draw[eb] (v11) -- (v12);
      \draw[er] (v6) -- (v13);
      \draw[eb] (v13) -- (v14);
      \draw[eg] (v14) -- (v15);
      \draw[ted] (v11) -- (v16);
      \draw[ted] (v16) -- (v17);
    \end{tikzpicture}
    \caption{$\ell=5$:
             $\chi(b;T')=(3,4,3)$.
             We thus colour $P_u$ periodically red--blue--green.}
    \label{fig:ub:l5}
  \end{subfigure}
  \begin{subfigure}[t]{0.04\textwidth}
  ~
  \end{subfigure}
  \begin{subfigure}[t]{0.29\textwidth}
    \centering
    \begin{tikzpicture}[scale=0.5]
      \clip (-0.5,-3.5) rectangle (7.5,2.5);
      \node[vertex] (v1) at (0,0) {};
      \node[vertex] (v2) at (1,0) {};
      \node[vertex] (v3) at (2,0) {};
      \node[vertex] (v4) at (3,0) {};
      \node[vertex] (v5) at (4,0) {};
      \node[vertex] (v6) at (5,0) {};
      \node[vertex] (v7) at (6,0) {};
      \node[vertex] (v8) at (1,-1) {};
      \node[vertex] (v9) at (1,-2) {};
      \node[vertex] (v10) at (3,-1) {};
      \node[vertex,label={180:$b$}] (v11) at (3,-2) {};
      \node[vertex] (v12) at (3,-3) {};
      \node[vertex] (v13) at (5,1) {};
      \node[vertex] (v14) at (5,2) {};
      \node[vertex] (v15) at (6,2) {};
      \node[vertex] (v16) at (4,-2) {};
      \node[vertex,label={0:$u$}] (v17) at (5,-2) {};
      \draw[er] (v1) -- (v2);
      \draw[eg] (v2) -- (v3);
      \draw[eb] (v3) -- (v4);
      \draw[er] (v4) -- (v5);
      \draw[eg] (v5) -- (v6);
      \draw[eb] (v6) -- (v7);
      \draw[eb] (v2) -- (v8);
      \draw[er] (v8) -- (v9);
      \draw[eg] (v4) -- (v10);
      \draw[er] (v10) -- (v11);
      \draw[eb] (v11) -- (v12);
      \draw[er] (v6) -- (v13);
      \draw[eb] (v13) -- (v14);
      \draw[eg] (v14) -- (v15);
      \draw[eg] (v11) -- (v16);
      \draw[eb] (v16) -- (v17);
    \end{tikzpicture}
    \caption{$\ell=6$:
             $\chi(b;T')=(5,3,4)$.
             We thus colour $P_u$ periodically green--blue--red.}
    \label{fig:ub:l6}
  \end{subfigure}

  \caption{Visualisation of the inductive $3$-colouring of a tree.
           The colour profiles are indexed red--green--blue.}
  \label{fig:ub}
\end{figure*}
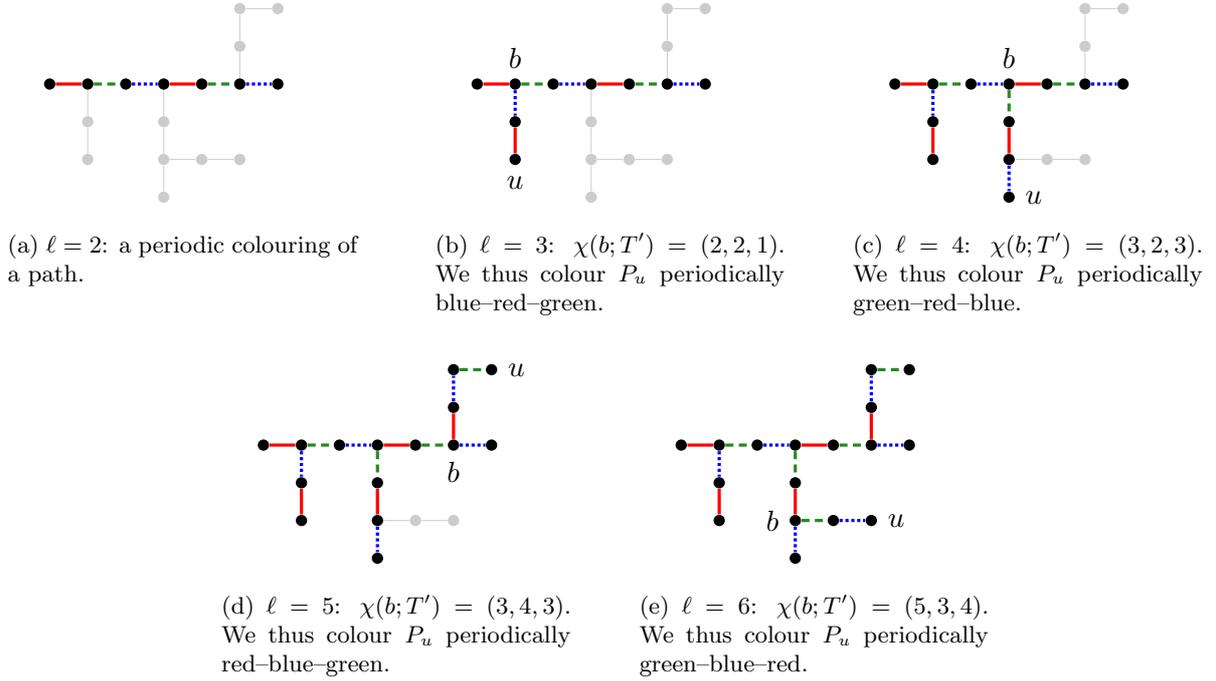

\subsection{Grids}\label{sec:grid}

We now prove \cref{cor:mclr:grid}.
\begin{proof}[Proof of \cref{cor:mclr:grid}]
  Let $G$ be the $2$-dimensional $m\times n$ grid ($m,n\ge 2$).
  Obtain $G^+$ from $G$ by adding a perfect matching covering the $4$ vertices of degree $2$ in $G$,
  so $\delta(G^+)=3$.
  Hence, by~\cite{KW91}*{Theorem 2}, $G^+$ has a spanning tree $T$
  with at least $mn/4+2$ leaves.
  By \cref{thm:mclr},
  \[
    \D_r(G^+,\cT)
    \ge \D_r(T,\cT)
    \ge \frac{r-1}{r}\cdot\left(\frac{mn}{4}+2\right)
    = \frac{r-1}{4r}\cdot mn +2-\frac{2}{r}
    \ge \frac{r-1}{4r}\cdot mn + 1.
  \]
  The result follows since
  $\D_r(G,\cT)\ge\D_r(G^+,\cT)-2Wr$.
\end{proof}

Let us show that this is indeed a strengthening (for $r=2$ and $m,n\ge 3$)
of \cite{BCJP20}*{Corollary~7}.
Assume that $m\le n$.
For $r=2$, \cref{cor:mclr:grid} shows that $\D_2(G,\cT)\ge mn/8-3$,
while \cite{BCJP20}*{Corollary~7} states that $\D_2(G,\cT)>mn/8-n/8-m$.
(In fact, in their proof, they actually give a slightly better bound on the discrepancy;
but our bound is still strictly better when $m\ge 5$.)

\subsection{Constructions}
In the introduction, we showed the lower bound of \cref{thm:mclr} is asymptotically tight. Here, we show the upper bound is (exactly) tight.
\begin{proposition}\label{prop:mclr:ub:sharp}
  For every $r,\ell\ge 2$ there exists a tree $T$ with $\ell(T)=\ell$
  and \[\D_r(T,\cT) = \ceil{(r-1)\cdot\frac{\ell}{2}}.\]
\end{proposition}

\begin{proof}
  A \defn{spider} is a star-like graph defined as follows:
  for $k\ge 1$ and $\ell\ge 2$,
  $\Sp^k_\ell$ is a tree with a root attached to $\ell$ paths (``legs''),
  each is of length $k$.
  Note that $\ell(\Sp^k_\ell)=\ell$.
  Let $T=\Sp^r_\ell$, and let $f:E(T)\to[r]$ be an $r$-colouring of its edges.
  We identify $E(T)$ by $[\ell]\times[r]$
  by labelling the $h$'th edge (counting from the root) of the $i$'th leg $(i,h)$.
  For $j\in[r]$ let $E_j\subseteq [\ell]\times[r]$ be the set of $j$-coloured edges in $T$.
  For $j\in[r]$, let $S_j$ be the smallest subtree of $T$ that contains the root
  and every $j$-coloured edge.
  Note that $w_j(S_j)\ge \sum_{(i,h)\in E_j}(r-h)$,
  thus
  \[
    \sum_{j\in[r]} w_j(S_j) \ge \sum_{(i,h)\in E(T)}(r-h) = \ell\binom{r}{2}.
  \]
  By the pigeonhole principle,
  there exists $j\in[r]$ for which $w_j(S_j)\ge \ceil{(r-1)\ell/2}$.
\end{proof}

\begin{figure*}[t!]
  \captionsetup{width=0.879\textwidth,font=small}
  \centering
  \begin{tikzpicture}
    \node[vertex] (root) at (0,0) {};
    \foreach \i [evaluate=\i as \a using \i*72,
                 evaluate=\i as \b using \i*72+12,
                 evaluate=\i as \c using \i*72+24
                 ] in {1,2,...,5} {
      \node[vertex] (u\i) at (\a:1) {};
      \node[vertex] (v\i) at (\b:1.707) {};
      \node[vertex] (w\i) at (\c:2.207) {};
      \draw (root) -- (u\i) -- (v\i) -- (w\i);
    }
  \end{tikzpicture}
  \caption{The spider $\Sp^3_5$.}
  \label{fig:spider}
\end{figure*}
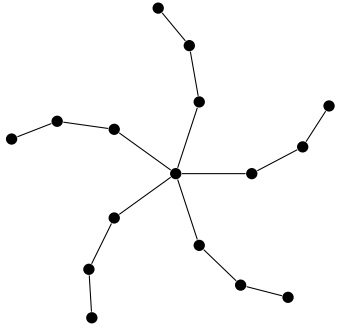

\section{Oriented discrepancy}\label{sec:orient}
In this section we prove \cref{thm:orient} and the sharpness of its lower bound.

\begin{proof}[Proof of the lower bound in \cref{thm:orient}]
  The oriented discrepancy of trees in an $\ell$-leaf star is $\ceil{\ell/2}+1$,
  see \cref{prop:odisc:star} below.
  We may therefore assume that $T$ is not a star.
  In particular,
  there exists a vertex $u$ which is not a leaf and has a neighbour which is also not a leaf.
  Consider an arbitrary orientation of the edges of $T$.
  Consider a $2$-colouring of $T$ according to the direction of each edge with respect to $u$:
  colour $e$ red if it is oriented towards $u$, and blue otherwise.
  By \cref{thm:mclr},
  there exist a subtree $T^*$ of $T$ with $2$-colour imbalance at least $\ceil{\frac{\ell}{2}}$.
  Assume $T^*$ maximises the $2$-colour imbalance.
  Assume further, without loss of generality, that the popular colour in $T^*$ is red.
  In particular, every edge from $T^*$ to its complement is blue.
  We claim that $u$ is in $T^*$, and is not a leaf of $T^*$.
  Indeed, if $u$ is not in $T^*$,
  let $w$ be the closest vertex to $u$ in $T^*$.
  Then, since the edge $\{w,z\}$ along the path from $w$ to $u$ is blue,
  the rooted tree $(T^*+z,w)$ has oriented imbalance at least $\ceil{\frac{\ell}{2}}+1$.
  Similarly, if $u$ is a leaf of $T^*$,
  then, since $u$ has a neighbour $v$ outside $T^*$, and the edge $\{u,v\}$ is blue,
  the rooted tree $(T^*+v,v)$ has oriented imbalance at least $\ceil{\frac{\ell}{2}}+1$.
  This shows, in particular, that all edges incident to $u$ are red.

  We conclude that every subtree of $T$ of maximal $2$-colour imbalance contains $u$ and its neighbourhood,
  and that all of these trees have the same popular colour (say, red).
  By the choice of $u$, it has a non-leaf neighbour $v$.
  Let $e=\{u,v\}$
  and consider the tree $T_1=T/e$ that is obtained from $T$ be contracting $e$
  and keeping the original orientations (and the induced $2$-colouring).
  Note that due to the choice of $v$, $u$ is not a leaf of $T_1$
  and $\ell(T_1)=\ell(T)=\ell$.
  Thus, applying \cref{thm:mclr} again
  yields a subtree $T_1^*$ of $T_1$ of maximal $2$-colour imbalance,
  which is at least $\ceil{\frac{\ell}{2}}$.
  By repeating the argument above
  (in which we did not assume that $u$ has a non-leaf neighbour, but only that it is not a leaf itself),
  we conclude that $u$ is in $T_1^*$.
  Let $T_2^*$ be obtained from $T_1^*$ be de-contracting $e$.
  If the dominant colour of $T_1^*$ is blue,
  then the rooted tree $(T_2^*,v)$ has oriented discrepancy at least $\ceil{\frac{\ell}{2}}+1$.
  If the dominant colour of $T_1^*$ is red,
  then the rooted tree $(T_2^*,u)$ has oriented discrepancy at least $\ceil{\frac{\ell}{2}}+1$.
\end{proof}

\begin{proof}[Proof of the upper bound in \cref{thm:orient}]
For a tree $T$ with a fixed orientation
we define $x_{v\rightarrow}^T,x_{v\leftarrow}^T$ to be the largest possible imbalance of a subtree rooted at $v$
in which the dominant orientation is from, respectively to, $v$.
We will prove that there exists an orientation of $T$
such that $x_{v\rightarrow}^T+x_{v\leftarrow}^T\leq \ell$ for all vertices $v$.
We prove the statement by induction on $\ell$.

For a tree with two leaves, i.e., a path, orienting the edges alternately clearly works.
Assume the statement holds for any tree with up to $\ell$ leaves and let $T$ be a tree with $\ell+1$ leaves.
Consider a leaf and the path connecting it to its branching vertex $v$.
Let $T'$ be the tree after removing this path.
By induction, $T'$ admits an orientation for which $x_{v\rightarrow}^{T'}+x_{v\leftarrow}^{T'}\leq \ell-1$ for all $v\in T'$. Assume, without loss of generality that $x_{v\to}^{T'}$ is the smallest of the two.
Now, orient the edges of the removed path alternately,
where the edge incident to $v$ is oriented towards it.
We claim that this orientation of $T$ satisfies the requirements.

First, for any $u\in V(T')$,
only $x_{u\leftarrow}^{T'}$ can possibly be increased by $1$,
so $x_{u\rightarrow}^T+x_{u\leftarrow}^T\leq \ell$.

Now for any vertex $u$ in the path, we have
\begin{itemize}
    \item $(x_{u\rightarrow}^T,x_{u\leftarrow}^T)\le(x_{v\rightarrow}^{T'}+2,x_{v\leftarrow}^{T'}-1)$ if $d(u,v)$ is odd. Note that the only case this is not true is when $x_{v\leftarrow}^{T'}=0$. But then, $0\leq x_{v\rightarrow}^{T'}\leq x_{v\leftarrow}^{T'}=0$, which is impossible, since at least one of them should be positive.
    Thus, $x_{u\rightarrow}^T+x_{u\leftarrow}^T \le \ell$.
    \item $(x_{u\rightarrow}^T,x_{u\leftarrow}^T)\le(x_{v\rightarrow}^{T'},x_{v\leftarrow}^{T'}+1)$ if $d(u,v)$ is even.
    Thus, $x_{u\rightarrow}^T+x_{u\leftarrow}^T \le \ell$.
\end{itemize}
This concludes the induction.
\end{proof}

\subsection{Constructions}
Here we show that the lower bound of \cref{thm:orient} is sharp.
\begin{proposition}\label{prop:odisc:star}
  Let $S_\ell$ be a star with $\ell\ge 2$ leaves.
  Then,
  \[
    \oD(S_\ell,\mathcal{DT}) = \ceil{\frac{\ell}{2}}+1.
  \]
\end{proposition}

\begin{proof}
  Consider any orientation of $S_\ell$.
  By the pigeonhole principle,
  at least $\ceil{\ell/2}$ edges are oriented in the same direction
  (towards the root or away from the root).
  Assume without loss of generality that at least $\ceil{\ell/2}$ edges are oriented away from the root.
  If at least $\ceil{\ell/2}+1$ edges are oriented away from the root,
  then the substar that consists of these edges and is rooted at the star's root
  has imbalance at least $\ceil{\ell/2}+1$.
  Otherwise, since $\ell>\ceil{\ell/2}$,
  there exists an edge that is oriented from a leaf $u$ towards the root.
  The substar that consists of that edge and all edges that are oriented away from the root,
  and is rooted at $u$,
  has imbalance $\ceil{\ell/2}+1$.
\end{proof}

\section{High dimensional discrepancy}\label{sec:highdim}
We begin by a short proof of \cref{prop:complex}.
\begin{proof}[Proof of \cref{prop:complex}]
  Let $\cH=(\cV,\cE)$ be a hypergraph,
  fix $r\ge 2$,
  and let $D=\D_r(\cH)$.
  Consider an $r$-colouring $g:\cV\to[r]$
  for which $\max_{A\in\cE} \left|r|g^{-1}(A)|-|A|\right|=D$.
  Let $\omega_j=\exp(2j\pi i/n)$, $j=1,...,r$, be the $r$'th roots of unity.
  Define $f:\cV\to\SS^1$ as follows: $f(a)=\omega_{g(a)}$.
  Now, given a hyperedge $A\in\cE$,
  write $A_j=g^{-1}(j)\cap A$ and $\alpha_j=|A_j|$.
  Write further $\beta_j=\alpha_j-|A|/r$.
  Note that
  \[
    f(A) = \sum_{j\in[r]} \sum_{a\in A_j} \omega_j
         = \sum_{j\in[r]} \alpha_j\omega_j
         = \frac{|A|}{r}\sum_{j\in[r]} \omega_j
         + \sum_{j\in[r]} \beta_j\omega_j
         = \sum_{j\in[r]} \beta_j\omega_j. 
  \]
  From the choice of $g$ it follows that $|r\alpha_j-|A||\le D$,
  hence $|\beta_j|\le D/r$.
  Thus, by the triangle inequality, $|f(A)|\le D$.
\end{proof}

We move on to prove \cref{thm:highdim:lb}.

\newcommand{\vc}[1]{\mathbf{#1}}
\begin{lemma}\label{lem:marginal}
  Let $d\ge 1$, and let $X=(X_1,\dots,X_d)\sim \Unif(\SS^{d-1})$. Then, the random variable $X_1$ is distributed on $[-1,1]$ with density function 
  \[
    f_{X_1}(x)=\frac{(1-x^2)^{\frac{d-3}{2}}}{\Beta\left(\frac{d-1}{2},\frac{1}{2}\right)}.
  \]
\end{lemma}
A proof of \cref{lem:marginal} can be found in \cite{185312}.

\begin{proof}[Proof of \cref{thm:highdim:lb}]
  Let $T$ be a tree with $\ell$ leaves,
  and let $f:E(T)\to\SS^d$ be a $d$-dimensional colouring its edges.
  Let $L\subseteq E(T)$ be the set of edges in $T$ that are incident to a leaf
  (so $|L|=\ell$).
  For a vector $\vc{v}\in\SS^d$
  let $L_{\vc{v}}$ be the set of edges $e$ in $L$ for which $\vc{v}\cdot f(e)>0$,
  and let $L'_{\vc{v}}=L\sm L_{\vc{v}}$.
  Denote by $T_{\vc{v}}$ the subtree of $T$ with the edge set $E(T)\sm L'_{\vc{v}}$
  and by $T'_{\vc{v}}$ the subtree of $T$ with the edge set $E(T)\sm L_{\vc{v}}$.
  For a subtree $S$ of $T$, write $D(S)=\sum_{e\in E(S)}f(e)$.
  Write $e_1=(1,0,\dots,0)\in\SS^d$ for the first vector in the standard basis.
  Let $\vc{v}$ be a uniformly random sampled vector in $\SS^d$,
  and set $D_{\vc{v}}=\sum_{e\in L}|\vc{v}\cdot f(e)|$.
  By linearity of expectation and by \cref{lem:marginal}
  \[
    \E{D_{\vc{v}}} = \ell\cdot\E|\vc{v}\cdot e_1|
          = 2\ell\cdot \int_{0}^1 \frac{x(1-x^2)^{d/2-1}}{\Beta\left(\frac{d}{2},\frac{1}{2}\right)}dx
          = \frac{2\ell}{d\cdot \Beta\left(\frac{d}{2},\frac{1}{2}\right)}.
  \]
  Thus, there exists a vector $\vc{v}$ for which $D_{\vc{v}} \ge \frac{2\ell}{d\cdot \Beta\left(\frac{d}{2},\frac{1}{2}\right)}$.
  By Cauchy--Schwarz we get
  \[
    \left\|D(T_{\vc{v}})-D(T'_{\vc{v}})\right\|
    \ge \left|\vc{v}\cdot \left(D(T_{\vc{v}})-D(T'_{\vc{v}})\right)\right|
    = D_{\vc{v}}
    \ge \frac{2\ell}{d\cdot \Beta\left(\frac{d}{2},\frac{1}{2}\right)}.
  \]
  By the triangle inequality, 
  $$D(S)\ge \frac{\ell}{d\cdot \Beta\left(\frac{d}{2},\frac{1}{2}\right)}$$
  for some $S\in\{T,T'\}$.
  A straightforward application of Stirling's formula yields the asymptotic bound as $d\rightarrow\infty,$ since
  \[
    \Beta\left(\frac{d}{2},\frac{1}{2}\right)\sim \sqrt{\pi}\cdot \left(\frac{d}{d+1}\right)^{\frac{d}{2}-\frac{1}{2}}
                \cdot \sqrt{\frac{2e}{d+1}}
    \sim \sqrt{\frac{2\pi}{d}}.\qedhere
  \]
\end{proof}

\paragraph{Acknowledgements}
The second author wishes to thank Boris Bukh, Matan Harel and Yinon Spinka 
for fruitful discussions at various stages of this project.

\bibliography{library}

\end{document}